\documentclass{amsart}
\usepackage{amsfonts,amssymb,amscd,amsmath,enumerate,verbatim,calc}
\usepackage[all]{xy}
\usepackage{stackrel}

\newcommand{\CM}{Cohen-Macaulay}

\newcommand{\Pb}{\mathbb{P} }

\newcommand{\m}{\mathfrak{m} }

\newcommand{\R}{\mathcal{R} }
\newcommand{\Z}{\mathbb{Z} }
\newcommand{\Sc}{\mathcal{S} }
\newcommand{\rt}{\rightarrow}

\newcommand{\ilim}{\varprojlim}

\DeclareMathOperator*{\limone}{lim^{\rlap{\scriptsize 1}}}

\newcommand{\coker}{\operatorname{coker}}

\newcommand{\depth}{\operatorname{depth}}

\newcommand{\projdim}{\operatorname{projdim}}
\newcommand{\Hom}{\operatorname{^*Hom}}
\newcommand{\Ext}{\operatorname{^*Ext}}
\newcommand{\Tor}{\operatorname{Tor}}

\theoremstyle{plain}

\newtheorem{theorem}{Theorem}[section]
\newtheorem{corollary}[theorem]{Corollary}

\theoremstyle{definition}

\newtheorem{s}[theorem]{}
\newtheorem{remark}[theorem]{Remark}

\theoremstyle{remark}

\begin{document}

\title[$\depth A/I^n$]{A lower bound on \textbf{depth} $\mathbf{A/I^n}$}
\author{Tony~J.~Puthenpurakal}
\date{\today}
\address{Department of Mathematics, Indian Institute of Technology Bombay, Powai, Mumbai 400 076, India}

\email{tputhen@gmail.com}

 \begin{abstract}
Let $(A,\mathfrak{m})$ be a Noetherian local ring and let $I\subseteq \mathfrak{m}$ be an ideal in $A$. Let $\ell(I)$ be the analytic spread of $I$. Let $G_I(A) = \bigoplus_{n \geq 0}I^n/I^{n+1}$ be the associated graded ring of $I$ and let its depth be $t$.  We prove 
$\depth A/I^n \geq t - \ell(I)$ for all $n\geq 1$.
 \end{abstract}
 \maketitle
\section{introduction}
Let $(A,\mathfrak{m})$ be a Noetherian local ring of dimension $d$ and let $I\subseteq \mathfrak{m}$ be an ideal in $A$. Let $G_I(A) = \bigoplus_{n \geq 0}I^n/I^{n+1}$ be the associated graded ring of $I$ and let $\R(I) = A[It]$ be the Rees algebra of $I$. Let $F(I) = \bigoplus_{n\geq 0}I^n/\m I^n$ be the fiber-cone of $I$. Recall that $\ell(I) = \dim F(I)$ is called the analytic spread of $I$.
A basic result due to Burch, see \cite[Corollary, p.\ 373]{Bu} is the following inequality
\[
\ell(I) \leq d - \min_{n} \depth A/I^n.
\]
Brodmann proved that $\lim_{n \rt \infty} \depth A/I^n$ exists, see \cite[Theorem 2(1)]{Br}. Furthermore he improved Burch's result by proving
\[
\ell(I) \leq d - \lim_{n \rt \infty} \depth A/I^n,
\]
see \cite[Theorem 2(2)]{Br}.

\s We were motivated by a result of Eisenbud and Huneke, see \cite[3.3]{EH} which states that
if $ I$ is  an ideal of $A$ of
height at least one and if $G_I(A)$ is Cohen-Macaulay then equality holds in Burch's result, i.e.,
\[
\ell(I) =  d - \min_{n} \depth A/I^n.
\]
 Note this also implies equality in Brodmann's result. So $\lim_{n \rt \infty} \depth A/I^n = d - \ell(I)$ when $G_I(A)$ is \CM \ and height $I$ is atleast one.

 \s In general Brodmann's result gives an upper bound
 $$\depth A/I^n \leq d - \ell(I) \quad \text{for all} \ n \gg 0. $$
 The objective of this paper is to give lower bounds on $\depth A/I^n$. Inspired by Eisenbud and Huneke's result we relate it to $\depth G_I(A)$. We prove:
 \begin{theorem}
 \label{bella}Let $(A,\mathfrak{m})$ be a Noetherian local ring and let $I\subseteq \mathfrak{m}$ be an ideal in $A$.  Let $\depth G_I(A) = t$.
 Then $\depth A/I^n \geq t - \ell(I)$ for all $n\geq 1$.
 \end{theorem}
 We note that our result has content only when $t > \ell(I)$.

\begin{remark}
  We note that $G_I(A)$ need not have a homogeneous system of parameters. This makes it impossible to prove Theorem \ref{bella} by induction on depth of $G_I(A)$.
\end{remark}

As an immediate corollary we obtain:
\begin{corollary}
  \label{cor}(with hypotheses as in \ref{bella}). We have 
$$ \depth G_I(A) \leq \ell(I) + \min_{m \geq 1} \{\depth A/I^n \}.$$
\end{corollary}
 \s\label{setup} \emph{Technique used to prove our result:} \\
 Set $L^I = \bigoplus_{n \geq 0}A/I^{n+1}$. In \cite[4.2]{P1} it is proved that $L^I$ is an $\R(I)$-module. Note that $L^I$ is not finitely generated as a $\R(I)$-module.
 Also note that we may assume that $A$ is complete with an infinite residue field. Set $\ell = \ell(I)$. Let $J = (x_1, \ldots, x_\ell)$ be a minimal reduction of $I$. Then $\R(I)$ is a finite $A[Jt]$-module.
We have a surjection $T = A[X_1, \ldots, X_{\ell}] \rt A[Jt]$ where $X_i$ is mapped to $x_it$. As $A$ is complete, we can choose a regular local ring $Q$ (say of dimension $q$) mapping surjectively to $A$. Set $\Sc = Q[X_1, \ldots, X_{\ell}]$. We have an obvious surjection $\Sc \rt T$. So $\R(I)$ is a finitely generated $\Sc$-module, We note that $\Sc$ is a $^*$-local regular ring of dimension $q + \ell$. So by Auslander-Buchsbaum theorem we have $\projdim_{\Sc} G_I(A) = q + \ell - t = e$. We prove
\begin{theorem}\label{rachel}
(with setup as in \ref{setup}). We have $\projdim_{\Sc} L^I = e$.
\end{theorem}
We note that if $E = \bigoplus_{n\in \Z}E_n$ is a graded projective $\Sc$-module (not necessarily finitely generated) then $E_n$ is a free $Q$-module for all $n \in \Z$ (use the fact that over local rings, all projective modules (not-necessarily finitely generated) are free, see\cite[2.5]{Mat}).

We give a proof of Theorem \ref{bella} assuming Theorem \ref{rachel}.
\begin{proof}
Let $\Pb$ be a graded, $\Sc$-projective resolution of $L^I$ of length $e$. Taking the degree $n^{th}$-component of this resolution we get a free resolution (not necessarily minimal) of $A/I^n$ as a $Q$-module.
So $\projdim_Q A/I^n \leq e$ for all $n \geq 1$. From this again by Auslander-Buchsbaum formula we get $\depth_Q A/I^n \geq t - \ell(I)$ for all $n\geq 1$. The result follows as $\depth_A A/I^n =
\depth_Q A/I^n$ (this is an exercise problem in \cite[1.2.26]{BH}).
\end{proof}
\begin{remark}
  The proof of Theorem \ref{rachel}, although short,  uses the functor $\ilim^1$ and so is difficult to guess.
\end{remark}
We now describe in brief the contents of this paper. In section two we discuss some preliminaries  that we need. In section three we prove Theorem \ref{rachel} which implies our result.

\section{Preliminaries}
We discuss a few preliminary results that we need.
Throughout $(A,\mathfrak{m})$ is a Noetherian local ring of dimension $d$ and $I\subseteq \mathfrak{m}$ is an ideal in $A$. Also $\R = \R(I)$ is the Rees algebra of $I$.
Set $G = G_I(A)$.

\s \label{l}(see \cite{P1}) Let $L = L^I = \bigoplus_{n \geq 0}A/I^{n+1}$. We have an inclusion of $\R$-algebras $\R \xrightarrow{i} A[It]$. We note that $\coker(i) = L(-1)$. So $L(-1)$ (and hence $L$) is an $\R$-module. Note that $L$ is \emph{not} finitely generated as an $\R$-module.
The exact sequence of $A$-modules $0 \rt I^n/I^{n+1} \rt A/I^{n+1} \rt A/I^n \rt 0$ can be assembled to get an exact sequence of $\R$-modules
$$ 0 \rt G \rt L \rt L(-1) \rt 0.$$

\s  For $r \geq 0$ consider the finitely generated $\R$
submodules $L_r $ of $L$ defined as follows:
\[
L_r = \left<\bigoplus_{n =0}^{r} \frac{A}{I^{n+1}} \right>.
\]
We have
\[
L_r = \bigoplus_{n =0}^{r} \frac{A}{I^{n+1}} \bigoplus \left( \bigoplus_{n \geq r +1}\frac{I^{n-r}}{I^{n+1}} \right).
\]
Notice that $L_0 = G$.  We note that $L_r \subseteq L_{r+1}$ for all $r \geq 0$. Notice
\[
L = \bigcup_{r \geq 0} L_r.
\]
For systemic reasons set $L_{-1} = 0$.
 For all $r \geq 0$ we have an exact sequence
\begin{equation}\label{Lr-eq}
  0 \rightarrow L_{r-1} \rightarrow L_r \xrightarrow{\rho_r} G(-r) \rightarrow 0.
\end{equation}
This filtration was introduced in \cite{P2}.
\section{Proof of Theorem \ref{rachel}}
In this section we prove Theorem \ref{rachel}. Set $G = G_I(A)$ and  $L = L^I$.
\begin{proof}
We first prove that $\projdim_{\Sc} L \leq e$.
It suffices to prove that if $M$ is a graded $\Sc$-module then $\Ext^{e+1}_\Sc(L, M) = 0$. As $\{L_r\}_{r\geq 0}$ is an ascending sequence of graded $\Sc$-submodules of $L$ whose union is $L$, we have an exact sequence for all $v\geq 0$,
\begin{equation*}
  0 \rt \limone_{\longleftarrow} \ \Ext^{v-1}_\Sc(L_r, M) \rt \Ext^{v}_\Sc(L, M) \rt \ilim \Ext^{v}_\Sc(L_r, M) \rt 0, \tag{*}
\end{equation*}
 see \cite[3.5.10]{W}.
For $r\ge 1$, by (\ref{Lr-eq}),  we have an exact sequence of $\Sc$-modules
$$ 0 \rt L_{r-1} \rt L_r \rt G(-r) \rt 0. $$
We apply the functor $\Hom(-, M)$ to the above exact sequence and look at the long exact sequence in cohomology. As $\projdim_{\Sc} G = e$ we obtain surjections
$\Ext^e_\Sc(L_r, M) \rt \Ext^e_\Sc(L_{r-1}, M)$ for all $r\geq 1$ and isomorphisms $\Ext^{e+1}_\Sc(L_r, M) \cong \Ext^{e + 1}_\Sc(L_{r-1}, M)$ for all $r \geq 1$.

The maps $\Ext^e_\Sc(L_r, M) \rt \Ext^e_\Sc(L_{r-1}, M)$ are surjections for $r \geq 1$. So we obtain $\ilim^1 \Ext^{e}_\Sc(L_r, M) = 0$, see \cite[3.5.3]{W}.

 For all $r \geq 1$, we have isomorphisms $\Ext^{e+1}_\Sc(L_r, M) \cong \Ext^{e + 1}_\Sc(L_{r-1}, M)$. So  we get for all $r\geq 1$
$$\Ext^{e+1}_\Sc(L_r, M) \cong \Ext^{e + 1}_\Sc(L_{0}, M) = \Ext^{e + 1}_\Sc(G, M) = 0.$$

 The assertion now follows from (*). Thus $\projdim_{\Sc} L \leq e$.

 Suppose if possible $\projdim_{\Sc} L <  e$. We have an exact sequence of $\Sc$-modules
 $$0 \rt G \rt L \rt L(-1) \rt 0, $$
 see \ref{l}.
 Let $k$ be the residue field of $Q$. We apply the functor $-\otimes_{\Sc} k$  to this exact sequence and consider the long exact  sequence in homology. We obtain $\Tor^{\Sc}_e(k, G) = 0$.
 So $\projdim_{\Sc}G < e$ which is a contradiction. Thus  $\projdim_{\Sc} L =  e$.
\end{proof}

\end{document}